\documentclass[preprint,12pt]{elsarticle}

\biboptions{numbers,sort&compress}
\usepackage{amssymb}
\usepackage{amsthm}
\usepackage{amsmath}
\usepackage{epic}
\usepackage{CJK}
\usepackage{setspace}%%行间距
\usepackage{bm}
\newtheorem{thm}{Theorem}[section]

\newtheorem{lem}[thm]{Lemma}

\newtheorem{defi}[thm]{Definition}

\newtheorem{rem}{Remark}
\journal{}

\begin{document}
\begin{spacing}{1.15}
\begin{CJK*}{GBK}{song}
\begin{frontmatter}
\title{\textbf{On a generalization of the spectral Mantel's theorem}}

\author[label1]{Chunmeng Liu}\ead{liuchunmeng0214@126.com}
\author[label1]{Changjiang Bu\corref{cor}}\ead{buchangjiang@hrbeu.edu.cn}
\cortext[cor]{Corresponding author}

\address{
\address[label1]{College of  Mathematical Sciences, Harbin Engineering University, Harbin 150001, PR China}
}

\begin{abstract}
Mantel's theorem is a classical result in extremal graph theory which implies that the maximum number of edges of a triangle-free graph of order $n$.
In 1970, E. Nosal obtained a spectral version of Mantel's theorem which gave the maximum spectral radius of a triangle-free graph of order $n$.
In this paper, we consider the generalized spectral extremal problems.
The clique tensor of a graph $G$ is proposed and the spectral Mantel's theorem is extended via the clique tensor.
Furthermore, a sharp upper bound of the number of cliques in $G$ via the spectral radius of the clique tensor is obtained.
And we show that the results of this paper implies that a result of Erd\H{o}s [Magyar Tud. Akad. Mat. Kutat\'{o} Int. K\"{o}zl. 7 (1962)] under certain conditions.
\end{abstract}

\end{frontmatter}

\section{Introduction}

The graphs in this paper are undirected and simple.
Let $G$ be a graph with the set of vertices $V(G)=\{1,2,\ldots,n\}$ and the set of edges $E(G)=\{e_{1},...,e_{m}\}$, the number of edges of $G$ is denoted by $e(G)$.
For a vertex $i\in V(G)$, the neighborhood of $i$ is the set of all vertices adjacent to $i$, denoted by $N(i)$.
A clique is a subset of vertices of a graph such that its induced subgraph is complete, and a clique of size $t$ is called a $t$-clique, denoted by $c_{t}$.
The set of all $t$-cliques in $G$ is denoted by $C_{t}(G)$ and the number of $t$-cliques of $G$ is denoted by $c_{t}(G)$.
Let $T_{r}(n)$ be the complete $r$-partite graph with partitions of sizes $n_{1}\leq n_{2}\leq\cdots\leq n_{r}$ and $n_{r}-n_{1}\leq 1$, the graph $T_{r}(n)$ is called \emph{$r$-partite Tur\'{a}n graph}.

Mantel's theorem, established by Willem Mantel \cite{Mantel} in 1907, is one of the early results in extremal graph theory and states that if $G$ is a triangle-free graph of order $n$, then $e(G)\leq e(T_{2}(n))$, equality holds if and only if $G=T_{2}(n)$.
In 1941, P. Tur\'{a}n \cite{Turan} extended the Mantel's theorem and obtained the well-known Tur\'{a}n theorem.
From then on, Tur\'{a}n type extremal problems have attracted extensive attention \cite{Fuedi}.

Let $\rho(G)$ be the spectral radius of the adjacency matrix of a graph $G$.
The spectral extremal problem is an important topic in the extremal graph theory.
Up to now, there have been many results on the spectral extremal graph theory \cite{spectral extremal surveys,survey on spectral conditions}.
In 1970, E. Nosal \cite{Nosal} obtained the maximum spectral radius of a triangle-free graph $G$ of order $n$ and the conclusion is called the spectral form of Mantel's theorem.

\begin{thm}(Spectral Mantel's theorem \textup{\cite{Nosal}})\label{thm Nosal}
Suppose that $G$ is a triangle-free graph of order $n$, then
\begin{align*}
\rho(G)\leq\rho(T_{2}(n)),
\end{align*}
equality holds if and only if $G=T_{2}(n)$.
\end{thm}

\begin{rem}\label{rem 1}
V. Nikiforov \cite{Nikiforov_1} proved that the relationship between Theorem \ref{thm Nosal} and Mantel's theorem.
Through the inequality $\frac{2e(G)}{n}\leq\rho(G)$ and Theorem \ref{thm Nosal}, for an $n$ vertices triangle-free graph $G$, then
\begin{align*}
e(G)\leq\left\lfloor\frac{n}{2}\rho(G)\right\rfloor\leq\left\lfloor\frac{n}{2}\rho(T_{2}(n))\right\rfloor
=\left\lfloor\frac{n}{2}\right\rfloor\left\lceil\frac{n}{2}\right\rceil=e(T_{2}(n)).
\end{align*}
It shows that Theorem \ref{thm Nosal} implies Mantel's theorem.
\end{rem}

V. Nikiforov \cite{Nikiforov_2} generalized Theorem \ref{thm Nosal} to $K_{r+1}$-free graphs and gave the spectral Tur\'{a}n's theorem.
\begin{thm}\textup{\cite{Nikiforov_2}}
If $G$ is a $K_{r+1}$-free graph of order $n$, then
\begin{align*}
\rho(G)\leq\rho(T_{r}(n)),
\end{align*}
equality holds if and only if $G=T_{r}(n)$.
\end{thm}
In addition, Theorem \ref{thm Nosal} is improved in terms of the number of edges of a graph in \cite{Nikiforov_3}.
For further results, see \cite{Ning Bo,Zhai,Peng}.

In this paper, the clique tensor of a graph is proposed and Theorem \ref{thm Nosal} is extended via the clique tensor.
Let $\mathbb{C}^{n}$ be the set of $n$-dimensional vectors over complex number field $\mathbb{C}$.
An order $m$ dimension $n$ complex tensor $\mathcal{A}=(a_{i_{1}i_{2}\cdots i_{m}})$ is a multidimensional array with $n^{m}$ entries, where $i_{j}=1,2,\ldots,n$, $j=1,2,\ldots,m$.
For vectors $x=(x_{1},\ldots,x_{n})^{T}\in\mathbb{C}^{n}$, the $\mathcal{A}x^{m-1}$ is a vector in $\mathbb{C}^{n}$ whose $i$-th component is $\sum_{i_{2},\ldots,i_{m}=1}^{n}a_{ii_{2}\cdots i_{m}}x_{i_{2}}\cdots x_{i_{m}}$ \cite{Qi}.
In 2005, Qi \cite{Qi} and Lim \cite{Lim} posed the eigenvalues of tensors, respectively.
If there exist $\lambda\in\mathbb{C}$ and nonzero vector $x=(x_{1},\ldots,x_{n})^{T}\in\mathbb{C}^{n}$ satisfy
\begin{align}\label{equ1}
\mathcal{A}x^{m-1}=\lambda x^{[m-1]},
\end{align}
then $\lambda$ is called an \emph{eigenvalue} of $\mathcal{A}$ and $x$ an \emph{eigenvector} of $\mathcal{A}$ corresponding to $\lambda$, where $x^{[m-1]}=(x_{1}^{m-1},\ldots,x_{n}^{m-1})^{T}$. The maximal modulus of all eigenvalues of $\mathcal{A}$ is called the spectral radius of $\mathcal{A}$, denoted by $\rho(\mathcal{A})$. 
The tensor and its eigenvalues have been applied to many fields, such as signal processing \cite{app_1}, automatica \cite{app_2}, polynomial optimization \cite{app_3}, network analysis \cite{app_4}, solving multilinear systems \cite{app_5,app_6}, spectral hypergraph theory \cite{app_7,app_8,app_9,app_10,app_11}, etc.
The authors \cite{me} improved the upper bound of Wilf's chromatic number in terms of the spectral radius of tensors and obtained a formula of the number of cliques of fixed size which involve the spectrum of tensors.

The \emph{$t$-clique tensor} of a graph $G$ is defined as follows.
\begin{defi}\label{defi cliques-tensor}
Let $G$ be a graph with $n$ vertices. An order $t$ and dimension $n$ tensor $\mathcal{A}(G)=(a_{i_{1}i_{2}\cdots i_{t}})$ is called the $t$-clique tensor of $G$, if
\begin{align*}
a_{i_{1}i_{2}\cdots i_{t}}=
\begin{cases}
   \frac{1}{(t-1)!}, &\{i_{1},\ldots,i_{t}\}\in C_{t}(G).\\
   0, &\textup{otherwise}.
\end{cases}
\end{align*}
\end{defi}

The eigenvalue of $\mathcal{A}(G)$ is called the \emph{$t$-clique eigenvalue} of $G$.
The spectral radius of $\mathcal{A}(G)$ is called the \emph{$t$-clique spectral radius} of $G$, denoted by $\rho_{t}(G)$.
When $t=2$, the $2$-clique tensor $\mathcal{A}(G)$ is the adjacency matrix of $G$ and $\rho_{2}(G)$ is the spectral radius of the adjacency matrix of $G$.

In this paper, Theorem \ref{thm Nosal} is generalized.

\begin{thm}\label{thm main thm 2}
Let $G$ be a graph on $n$ vertices. If $G$ is $K_{r+1}$-free, then
\begin{align*}
\rho_{r}(G)\leq \rho_{r}(T_{r}(n)),
\end{align*}
equality holds if and only if $G$ is the $r$-partite Tur\'{a}n graph $T_{r}(n)$.
\end{thm}
Obviously, Theorem \ref{thm main thm 2} is the spectral Mantel's theorem in the case when $r=2$.
Similar to Remark \ref{rem 1}, a natural question is that does Theorem \ref{thm main thm 2} imply the maximum number of $r$-cliques in $K_{r+1}$-free graphs of order $n$?
Therefore, we study the maximum number of cliques of fixed size in graphs in terms of spectral radius of tensors and obtain a following sharp upper bound of $c_{r}(G)$ in terms of $\rho_{r}(G)$.

\begin{thm}\label{thm main thm 1}
For the $r$-clique spectral radius $\rho_{r}(G)$ of a graph $G$ on $n$ vertices, then
\begin{align}\label{equ p-spectral radio 1}
c_{r}(G)\leq\frac{n}{r}\rho_{r}(G).
\end{align}
Furthermore, if the number of $r$-cliques containing $i$ is equal for all $i\in V(G)$, then equality holds in \textup{(\ref{equ p-spectral radio 1})}.
\end{thm}

For a $K_{r+1}$-free graph $G$ of order $n$, P. Erd\H{o}s \cite{Erdos} proved that
\begin{align}\label{equ Erdos}
c_{r}(G)\leq c_{r}(T_{r}(n))
\end{align}
with equality holds if and only if $G$ is the $r$-partite Tur\'{a}n graph $T_{r}(n)$, and \cite{Zykov,Simonovits} gave the same conclusion through different methods.
In Section 3, we will show that Theorem \ref{thm main thm 2} implies (\ref{equ Erdos}) in the case when $r\mid n$.
The remainder of the paper is organized as follows.
In Section 2, we introduce some of the definitions and lemmas required for proofs.
In Section 3, we give the proofs of Theorem \ref{thm main thm 2} and Theorem \ref{thm main thm 1}.

\section{Preliminaries}

The tensor $\mathcal{A}$ is called symmetric if its entries are invariant under any permutation of their indices.
If all elements of a tensor $\mathcal{A}$ are nonnegative, then $\mathcal{A}$ is called the nonnegative tensor.
For the symmetric nonnegative tensor, Qi \cite{Qi2013} gave the following conclusion.

\begin{lem} \label{Thm Qi2013} \textup{\cite{Qi2013}}
Suppose that $\mathcal{A}$ is an order $m$ dimension $n$ symmetric nonnegative tensor, with $m\geq2$. Then
\begin{align*}
\rho(\mathcal{A})=\max\left\{x^{T}\mathcal{A}x^{m-1}: \sum_{i=1}^{n}x_{i}^{m}=1, x=(x_{1},\ldots,x_{n})^{T}\in\mathbb{R}_{+}^{n}\right\},
\end{align*}
where $\mathbb{R}_{+}^{n}$ is the set of $n$-dimensional nonnegative real vectors.
\end{lem}

The upper and lower bounds of spectral radius of nonnegative tensors are obtained in \cite{Yang}.
\begin{lem} \label{lem Yang 1} \textup{\cite{Yang}}
Let $\mathcal{A}=(a_{i_{1}i_{2}\cdots i_{m}})$ be an order $m$ dimension $n$ nonnegative tensor. Then
\begin{align*}
\min_{1\leq i\leq n}\sum_{i_{2},\ldots,i_{m}=1}^{n}a_{ii_{2}\cdots i_{m}}\leq\rho(\mathcal{A})\leq\max_{1\leq i\leq n}\sum_{i_{2},\ldots,i_{m}=1}^{n}a_{ii_{2}\cdots i_{m}}.
\end{align*}
\end{lem}

For an order $m$ dimension $n$ nonnegative tensor $\mathcal{A}=(a_{i_{1}i_{2}\cdots i_{m}})$, let $G_{\mathcal{A}}=(V(G_{\mathcal{A}}),E(G_{\mathcal{A}}))$ be the digraph of the tensor $\mathcal{A}$ with vertex set $V(G_{\mathcal{A}})=\{1,2,\ldots,n\}$ and arc set $E(G_{\mathcal{A}})=\{(i,j)|a_{ii_{2}\cdots i_{m}}\neq0, j\in\{i_{2},\ldots,i_{m}\}\}$.
The nonnegative tensor $\mathcal{A}$ is weakly irreducible if the corresponding directed graph $G_{\mathcal{A}}$ is strongly connected. Otherwise, the tensor $\mathcal{A}$ is weakly reducible (see \cite{Friedland}).
The Perron-Frobenius theorem for weakly irreducible nonnegative tensors is given in \cite{Friedland}.

\begin{lem}{\rm \cite{Friedland}} \label{thm Perron Th.}
If a nonnegative tensor $\mathcal{A}$ is weakly irreducible, then $\rho(\mathcal{A})$ is the unique positive eigenvalue of $\mathcal{A}$, with the unique positive eigenvector $x$, up to a positive scaling coefficient.
\end{lem}

Let $\mathcal{A}$ and $\mathcal{B}$ be two order $m$ dimension $n$ nonnegative tensors. If $\mathcal{B}-\mathcal{A}$ is nonnegative, we write $\mathcal{A}\leq\mathcal{B}$.

\begin{lem}\textup{\cite{Fan}}\label{lem A<B}
Suppose $0\leq\mathcal{A}\leq\mathcal{B}$, then $\rho(\mathcal{A})\leq\rho(\mathcal{B})$. Furthermore, if $\mathcal{B}$ is weakly irreducible and $\mathcal{A}\neq\mathcal{B}$, then $\rho(\mathcal{A})<\rho(\mathcal{B})$.
\end{lem}

Shao et al \cite{Shao} obtained the following result on the ``weakly reducible canonical form'' which is a generalization of the corresponding ``reducible canonical form'' theorem for matrices.
\begin{lem}\textup{\cite{Shao}}\label{lem diagonal block}
Let $\mathcal{A}$ be an order $m$ dimension $n$ tensor.
Then there exists positive integers $r\geq1$ and $n_{1},\ldots,n_{r}$ with $n_{1}+\cdots+n_{r}=n$ such that $\mathcal{A}$ is permutational similar to some $(n_{1},\ldots,n_{r})$-lower triangular block tensor, where all the diagonal blocks $\mathcal{A}_{1},\ldots,\mathcal{A}_{r}$ are weakly irreducible.
Furthermore, we have $\rho(\mathcal{A})=\max\{\rho(\mathcal{A}_{1}),\ldots,\rho(\mathcal{A}_{r})\}$.
\end{lem}

\section{Proofs of Theorems}

A walk is a sequence of edges $e_{1},e_{2},\ldots,e_{m}$ in which $e_{i}$ and $e_{i+1}$ are incident with a common vertex for $i=1,\ldots,m-1$, and a graph is connected if any two of its vertices are joined by a walk.
The adjacency matrix of a graph $G$ is irreducible if and only if $G$ is connected (see \cite{Sachs}).
A \emph{$t$-clique walk} of a graph $G$ is proposed to determine the $t$-clique tensor of $G$ is weakly irreducible, that is,
a sequence of $t$-cliques $c_{t}^{1},c_{t}^{2},\ldots,c_{t}^{m}$ in which $c_{t}^{i}$ and $c_{t}^{i+1}$ have at least one vertex in common for $i=1,\ldots,m-1$.
A graph $G$ is called \emph{$t$-clique connected} if any two of its vertices are joined by a $t$-clique walk.
For the $t$-clique tensor of a graph $G$, we have the following conclusion.

\begin{lem}\label{lem t-clique tensor weakly irr.}
The $t$-clique tensor of a graph $G$ is weakly irreducible if and only if $G$ is $t$-clique connected.
\end{lem}
\begin{proof}
Let $\mathcal{A}(G)=(a_{i_{1}i_{2}\cdots i_{t}})$ be the $t$-clique tensor of $G$. Firstly, we prove that $\mathcal{A}(G)$ is weakly irreducible if $G$ is $t$-clique connected.

Let $i_{1},i_{2},\ldots,i_{t}$ be the $t$ vertices of $G$.
Suppose that vertices $i_{1},i_{2},\ldots,i_{t}$ compose a $t$-clique in $G$, we obtain
$a_{i_{1}i_{2}\cdots i_{t}}=a_{\sigma(i_{1}i_{2}\cdots i_{t})}=\frac{1}{(t-1)!}$ by Definition \ref{defi cliques-tensor}, where $\sigma$ denotes a permutation of $\{i_{1},i_{2},\ldots,i_{t}\}$.
By the definition of the digraph $G_{\mathcal{A}(G)}$ of $\mathcal{A}(G)$, vertices $i_{1},i_{2},\ldots,i_{t}$ form a complete digraph of $t$ vertices in $G_{\mathcal{A}(G)}$.
Since $G$ is $t$-clique connected, any two of its vertices are joined by a $t$-clique walk.
Thus, we obtain $G_{\mathcal{A}(G)}$ is strongly connected imply that the $t$-clique tensor $\mathcal{A}(G)$ of $G$ is weakly irreducible.

Next, we prove that $\mathcal{A}(G)$ is weakly reducible if $G$ is not $t$-clique connected.

Since $G$ is not $t$-clique connected, there exist two vertices $i$ and $j$ can not join by any $t$-clique walks.
Therefore, vertices $i$ and $j$ can not join by any walks in the digraph $G_{\mathcal{A}(G)}$ of $\mathcal{A}(G)$, implying that $G_{\mathcal{A}(G)}$ is not strongly connected. That is, the $t$-clique tensor $\mathcal{A}(G)$ is weakly reducible.
\end{proof}

The proof of Theorem \ref{thm main thm 2} is divided into three parts.
Let $G'$ be a $K_{r+1}$-free graph on $n$ vertices with the largest $r$-clique spectral radius.
Firstly, we show that $G'$ is $r$-clique connected.
Secondly, we prove that $G'$ is a complete $r$-partite graph.
Lastly, for all the complete $r$-partite graphs on $n$ vertices, we show that the graph with the maximum $r$-clique spectral radius is the $r$-partite Tur\'{a}n graph $T_{r}(n)$.

\begin{lem}\label{lem c. connected}
The graph $G'$ is $r$-clique connected.
\end{lem}
\begin{proof}
Let $\mathcal{A}(G')$ be the $r$-clique tensor of $G'$.
Suppose to the contrary that $G'$ is not $r$-clique connected, then $\mathcal{A}(G')$ is weakly reducible by Lemma \ref{lem t-clique tensor weakly irr.}.
Since $\mathcal{A}(G')$ is weakly reducible, there exists a subgraph $H$ of $G'$ such that $H$ is an $r$-clique component (an $r$-clique connected subgraph that is not part of any larger $r$-clique connected subgraph) attaining the $r$-clique spectral radius of $G'$ by Lemma \ref{lem diagonal block}, that is $\rho_{r}(G')=\rho_{r}(H)$.
The graph $G'$ is not $r$-clique connected also implies that there exist two vertices $i$ and $j$ can not join by any $r$-clique walks. For the vertex $i$, we discuss it in two cases.

Case 1: The vertex $i\in V(H)$.
Then the vertex $j\notin V(H)$, that is, the vertex $j$ does not form an $r$-clique with vertices in $H$.
Therefore, the vertex $j$ is adjacent to at most $r-2$ vertices of any $r$-clique in $H$.
Take an $r$-clique $c^{\prime}_{r}=\{v_{1},\ldots,v_{r}\}$ of $H$ with the most vertices adjacent to $j$.
Suppose that vertices $v_{1},\ldots,v_{s}$ are adjacent to $j$, we have $s\leq r-2$.
Let $H'$ be the graph with $V(H')=V(H)\cup\{j\}$ and obtained by adding edges between the vertices $v_{s+1},\ldots,v_{r-1}$ and $j$.
Thus, the vertices $v_{1},\ldots,v_{r-1},j$ form an $r$-clique in $H'$.
Since $G'$ is $K_{r+1}$-free, the graph $H$ is $K_{r+1}$-free implies that $H'$ is also $K_{r+1}$-free.
Otherwise, there exists a vertex $k$ such that $v_{1},\ldots,v_{r-1},j,k$ form an $(r+1)$-clique in $H'$.
That is, the vertices $v_{1},\ldots,v_{r-1},k$ form an $r$-clique in $H$ and $j$ is adjacent to $v_{1},\ldots,v_{s},k$.
This contradicts the choice of $c^{\prime}_{r}$.
The graph $H'$ is $r$-clique connected and $H$ is an $r$-clique connected subgraph of $H'$ with $|V(H)|<|V(H')|$.
Add an isolated vertex to $H$ such that the resulting graph $\hat{H}$ has the same number of vertices as $H'$, and we have $\rho_{r}(H)=\rho_{r}(\hat{H})$.
Since $H'$ is $r$-clique connected, the $r$-clique tensor $\mathcal{A}(H')$ of $H'$ is weakly irreducible by Lemma \ref{lem t-clique tensor weakly irr.}.
The $r$-clique tensor $\mathcal{A}(\hat{H})<\mathcal{A}(H')$ implies that $\rho_{r}(H)=\rho_{r}(\hat{H})<\rho_{r}(H')$ by Lemma \ref{lem A<B}.
Thus $\rho_{r}(G')=\rho_{r}(H)<\rho_{r}(H')$, this contradicts the choice of $G'$.

Case 2: The vertex $i\notin V(H)$.
Similar to Case 1, we can construct a $K_{r+1}$-free graph $H''$ through the vertex $i$ and $H$ such that $\rho_{r}(G')=\rho_{r}(H)<\rho_{r}(H'')$. It is also contradictory.
\end{proof}

\begin{lem}
The graph $G'$ is a complete $r$-partite graph.
\end{lem}
\begin{proof}
Let $\mathcal{A}(G')$ be the $r$-clique tensor of $G'$.
From Lemma \ref{lem c. connected}, the graph $G'$ is $r$-clique connected.
Thus, the $r$-clique tensor $\mathcal{A}(G')$ is weakly irreducible by Lemma \ref{lem t-clique tensor weakly irr.}.
By Lemma \ref{thm Perron Th.}, the $r$-clique spectral radius $\rho_{r}(G')$ is the unique positive eigenvalue of $\mathcal{A}(G')$, with the unique positive eigenvector $x=(x_{1},\ldots,x_{n})^{T}$. For the $r$-clique spectral radius $\rho_{r}(G')$, we have
\begin{align*}
\rho_{r}(G')=x^{T}\mathcal{A}(G')x^{r-1}=r\sum_{\{i_{1},\ldots,i_{r}\}\in C_{t}(G')}x_{i_{1}}\cdots x_{i_{r}}.
\end{align*}
In order to prove that $G'$ is a complete $r$-partite graph, suppose to the contrary that $G'$ is not a complete $r$-partite graph which implies that $G'$ is a complete $s$-partite graph for $s<r$ or there exist two non adjacent vertices $i,j\in V(G')$ such that $N(i)\neq N(j)$. We discuss it in the following two cases.

Case 1. The graph $G'$ is a complete $s$-partite graph for $s<r$, then the $r$-clique does not exist in $G'$.
This contradicts that $G'$ is $r$-clique connected.

Case 2. There exist two non adjacent vertices $i,j\in V(G')$ such that $N(i)\neq N(j)$, suppose that there exists a vertex $k\in V(G')$ adjacent to $i$ and not adjacent to $j$.
For vertex $v\in V(G')$ and the eigenvector $x=(x_{1},\ldots,x_{n})^{T}$ corresponding to $\rho_{r}(G')$, we obtain
\begin{align*}
\rho_{r}(G')x_{v}^{r-1}=\sum_{\{v,i_{2},\ldots,i_{r}\}\in C_{t}(G')} x_{i_{2}}\cdots x_{i_{r}}.
\end{align*}
Let $WS_{G'}(v,x)=\sum_{\{v,i_{2},\ldots,i_{r}\}\in C_{t}(G')} x_{i_{2}}\cdots x_{i_{r}}$.
For $WS_{G'}(i,x)$, $WS_{G'}(j,x)$ and $WS_{G'}(k,x)$, it will be discussed on two cases.

Case \romannumeral1. $WS_{G'}(j,x)<WS_{G'}(i,x)$ or $WS_{G'}(j,x)<WS_{G'}(k,x)$.

If $WS_{G'}(j,x)<WS_{G'}(i,x)$, we delete all edges incident to $j$ and connect $j$ to all vertices in $N(i)$.
After such a graph operation, we get a new graph $G''$, and $G''$ is still $K_{r+1}$-free.
For the graph $G''$ and the positive eigenvector $x=(x_{1},\ldots,x_{n})^{T}$ corresponding to $\rho_{r}(G')$, by Lemma \ref{Thm Qi2013}, we have
\begin{align*}
\rho_{r}(G'') &\geq r\sum_{\{i_{1},\ldots,i_{r}\}\in C_{t}(G'')}x_{i_{1}}\cdots x_{i_{r}}\\
&=r\sum_{\{i_{1},\ldots,i_{r}\} \in C_{t}(G')}x_{i_{1}}\cdots x_{i_{r}}-rx_{j}WS_{G'}(j,x)+rx_{j}WS_{G'}(i,x)\\
&>r\sum_{\{i_{1},\ldots,i_{r}\} \in C_{t}(G')}x_{i_{1}}\cdots x_{i_{r}}=\rho_{r}(G').
\end{align*}
This contradicts the choice of $G'$. When $WS_{G'}(j,x)<WS_{G'}(k,x)$, the proof is similar.

Case \romannumeral2. $WS_{G'}(j,x)\geq WS_{G'}(i,x)$ and $WS_{G'}(j,x)\geq WS_{G'}(k,x)$.

In this case, we delete all edges incident to $i$ and connect $i$ to all vertices in $N(j)$.
Delete all edges incident to $k$ and connect $k$ to all vertices in $N(j)$.
Then we get a graph $G'''$ and $G'''$ is also a $K_{r+1}$-graph.
Similar to Case \romannumeral1, we have
\begin{align*}
\rho_{r}(G''') &\geq r\sum_{\{i_{1},\ldots,i_{r}\} \in C_{t}(G''')}x_{i_{1}}\cdots x_{i_{r}}\\
&=r\sum_{\{i_{1},\ldots,i_{r}\} \in C_{t}(G')}x_{i_{1}}\cdots x_{i_{r}}-rx_{i}WS_{G'}(i,x)-rx_{k}WS_{G'}(k,x)\\
&+rx_{i}WS_{G'}(j,x)+rx_{k}WS_{G'}(j,x)+r\sum_{\{i,k,i_{3},\ldots,i_{r}\}\in C_{t}(G')}x_{i}x_{k}x_{i_{3}}\cdots x_{i_{r}}.
\end{align*}
Since $i$ and $k$ are adjacency in $G'$ and each edge of $G'$ is contained in an $r$-clique, we obtain
\begin{align*}
r\sum_{\{i,k,i_{3},\ldots,i_{r}\}\in C_{t}(G')}x_{i}x_{k}x_{i_{3}}\cdots x_{i_{r}}>0.
\end{align*}
Hence, we have $\rho_{r}(G''')>\rho_{r}(G')$. This also contradicts the choice of $G'$.

To sum up, the graph $G'$ is a complete $r$-partite graph.
\end{proof}

\begin{lem}\label{lem r-clique spectral radius}
For the $r$-clique spectral radius of the complete $r$-partite graph on $n$ vertices, we have
\begin{align*}
\rho_{r}(T_{r}(n))=\max\{\rho_{r}(H):\textup{$H$ is a complete $r$-partite graph on $n$ vertices}\},
\end{align*}
where $T_{r}(n)$ is the $r$-partite Tur\'{a}n graph.
\end{lem}
\begin{proof}
Let $H$ be the complete $r$-partite graph on $n$ vertices with partitions of sizes $n_{1},n_{2},\ldots,n_{r}$.
And let $\rho_{r}(H)$ be the $r$-clique spectral radius of $H$ and $x=(x_{1},\ldots,x_{n})^{T}$ be the eigenvector corresponding to $\rho_{r}(H)$.
Since $H$ is $r$-clique connected, the $r$-clique tensor $\mathcal{A}(H)$ is weakly irreducible.
By Lemma \ref{thm Perron Th.}, we obtain $\rho_{r}(H)>0$ with the positive eigenvector $x=(x_{1},\ldots,x_{n})^{T}$.
The vertices in each part of $H$ have the same neighborhood implies that the components of the eigenvectors corresponding to these vertices are equal.
Let components of the eigenvectors corresponding to the vertices in each part of $H$ be $x_{n_{s}}$, $s=1,2,\ldots,r$.
Hence, for the characteristic equations of $\rho_{r}(H)$, we have
\begin{align}\label{equ characteristic equations}
\left\{
\begin{aligned}
&\rho_{r}(H)x_{n_{1}}^{r-1}=n_{2}n_{3}\cdots n_{r}x_{n_{2}}x_{n_{3}}\cdots x_{n_{r}}\\
&\rho_{r}(H)x_{n_{2}}^{r-1}=n_{1}n_{3}\cdots n_{r}x_{n_{1}}x_{n_{3}}\cdots x_{n_{r}}\\
&\cdots\\
&\rho_{r}(H)x_{n_{r}}^{r-1}=n_{1}n_{2}\cdots n_{r-1}x_{n_{1}}x_{n_{2}}\cdots x_{n_{r-1}}
\end{aligned}
\right.
\end{align}
Multiply $r$ equations in (\ref{equ characteristic equations}), we get
\begin{align*}
(\rho_{r}(H))^{r}x_{n_{1}}^{r-1}x_{n_{2}}^{r-1}\cdots x_{n_{r}}^{r-1}=(n_{1}n_{2}\cdots n_{r})^{r-1}x_{n_{1}}^{r-1}x_{n_{2}}^{r-1}\cdots x_{n_{r}}^{r-1}.
\end{align*}
Since $x_{n_{1}},x_{n_{2}},\ldots,x_{n_{r}}$ are all positive, the $r$-clique spectral radius $\rho_{r}(H)=(n_{1}n_{2}\cdots n_{r})^{\frac{r-1}{r}}$.
Thus, the $r$-clique spectral radius $\rho_{r}(H)$ is maximal if and only if the $n_{1},n_{2},\ldots,n_{r}$ are as equal as possible, that is $H\simeq T_{r}(n)$.
\end{proof}

Next, we give the proof of Theorem \ref{thm main thm 1} in the following.

\begin{proof}[Proof of Theorem \ref{thm main thm 1}]

By the definition of symmetric tensor, the entries of the $r$-clique tensor $\mathcal{A}(G)$ of $G$ are invariant under any permutation of their indices, thus $\mathcal{A}(G)$ is a symmetric tensor. For the $r$-clique tensor $\mathcal{A}(G)=(a_{i_{1}i_{2}\cdots i_{r}})$, we have
\begin{align}\label{equ the sum of k_p+1 tensor}
\sum_{i_{2},\ldots,i_{r}=1}^{n} a_{ii_{2}\cdots i_{r}} \ \textup{equals the number of $r$-cliques containing $i$} \ (i=1,\ldots,n).
\end{align}

Let $\rho_{r}(G)$ be the spectral radius of the $r$-clique tensor $\mathcal{A}(G)$.
Since $\mathcal{A}(G)$ is a symmetric nonnegative tensor, and by Lemma \ref{Thm Qi2013}, we obtain
\begin{align}\label{equ Rayleigh quotient}
\rho_{r}(G)=\max\left\{x^{T}\mathcal{A}x^{r-1}: \sum_{i=1}^{n}x_{i}^{r}=1, x=(x_{1},\ldots,x_{n})^{T}\in\mathbb{R}_{+}^{n}\right\}.
\end{align}

Let $x_{i}=n^{-\frac{1}{r}} \ (i=1,\ldots,n)$ in (\ref{equ Rayleigh quotient}). Then

\begin{align*}
\rho_{r}(G) &\geq x^{T}\mathcal{A}(G)x^{r-1}=\sum_{i_{1},\ldots,i_{r}}^{n} a_{i_{1}\cdots i_{r}}x_{i_{1}}\cdots x_{i_{r}}\\
&=\frac{\sum_{i=1}^{n}\sum_{i_{2},\ldots,i_{r}=1}^{n} a_{ii_{2}\cdots i_{r}}}{n}\\
&=\frac{r\cdot c_{r}(G)}{n}.
\end{align*}
Therefore, the inequality \textup{(\ref{equ p-spectral radio 1})} is obtained.

If the number of $r$-cliques containing $i$ is equal for all $i\in V(G)$, and by (\ref{equ the sum of k_p+1 tensor}), then

\begin{align*}
\sum_{i_{2},\ldots,i_{r}=1}^{n} a_{ii_{2}\cdots i_{r}}=\frac{r\cdot c_{r}(G)}{n} \ (i=1,\ldots,n).
\end{align*}

For the spectral radius $\rho_{r}(G)$, by Lemma \ref{lem Yang 1},  we have

\begin{align*}
\rho_{r}(G)=\frac{r\cdot c_{r}(G)}{n}.
\end{align*}
\end{proof}

Next, we show that the relation between Theorem \ref{thm main thm 2} and inequality (\ref{equ Erdos}).
Through Theorem \ref{thm main thm 2} and Theorem \ref{thm main thm 1}, for an $n$ vertices $K_{r+1}$-free graph $G$, we obtain
\begin{align*}
c_{r}(G)\leq\left\lfloor\frac{n}{r}\rho_{r}(G)\right\rfloor\leq\left\lfloor\frac{n}{r}\rho_{r}(T_{r}(n))\right\rfloor
=\left\lfloor\frac{n}{r}\left(\prod_{s=0}^{r-1}\left\lfloor\frac{n+s}{r}\right\rfloor\right)^{\frac{r-1}{r}}\right\rfloor.
\end{align*}
When $r\mid n$, we have $\left\lfloor\frac{n}{r}\left(\prod_{s=0}^{r-1}\left\lfloor\frac{n+s}{r}\right\rfloor\right)^{\frac{r-1}{r}}\right\rfloor
=\left(\frac{n}{r}\right)^r=c_{r}(T_{r}(n))$.
Thus, Theorem \ref{thm main thm 2} implies (\ref{equ Erdos}) in the case when $r\mid n$.
When $r \nmid n$, the following example is considered.
For $3$-partite Tur\'{a}n graph with 28 vertices, the number of triangles in $T_{3}(28)$ is
$c_{3}(T_{3}(28))=9\times9\times10=810<
\left\lfloor\frac{28}{3}\left(\prod_{s=0}^{2}\left\lfloor\frac{28+s}{3}\right\rfloor\right)^{\frac{2}{3}}\right\rfloor=811$.

\section*{Acknowledgement}

This work is supported by the National Natural Science Foundation of China (No. 11801115, No. 12071097 and No. 12042103), the Natural Science Foundation of the Heilongjiang Province (No. QC2018002) and the Fundamental Research Funds for the Central Universities.

\end{CJK*}
\end{spacing}
\end{document}